\documentclass[letterpaper, 10 pt, conference]{ieeeconf}  % Comment this line out
\IEEEoverridecommandlockouts                              % This command is only
\usepackage{graphicx}
\usepackage{bm}
\usepackage{color}
\newtheorem{theorem}{Theorem}
\newtheorem{lemma}{Lemma}

\newtheorem{remark}{Remark}

\newcommand{\E}{\mathbb{E} }

\usepackage{amsmath} % assumes amsmath package installed
\usepackage{amssymb}  % assumes amsmath package installed

\title{\LARGE \bf
Mean Square Capacity of Power Constrained Fading Channels with Causal Encoders and Decoders*
}

\author{Liang Xu$^{1}$, Lihua Xie$^{1}$ and Nan Xiao$^{2}$% <-this % stops a space
\thanks{*This work was supported by the National Research Foundation of Singapore under Grant NRF2011NRF$\textrm{-}$CRP001$\textrm{-}$090 and the National Natural Science Foundation of China under Grant 61304044.}% <-this % stops a space
\thanks{$^{1}$Liang Xu and Lihua Xie are with EXQUISITUS, Centre for E-City, School of Electrical and Electronic Engineering, Nanyang Technological University, Singapore 639798, Singapore
        {\tt\small lxu006@e.ntu.edu.sg, elhxie@ntu.edu.sg}}%
\thanks{$^{2}$Nan Xiao is with the Singapore-MIT Alliance for Research and Technology Centre, Singapore 138602, Singapore
        {\tt\small xiaonan@smart.mit.edu}}%
}

\begin{document}
\maketitle
\thispagestyle{empty}
\pagestyle{empty}

\begin{abstract}
This paper is concerned with the mean square stabilization problem of discrete-time LTI systems over a power constrained fading channel. Different from existing research works, the channel considered in this paper suffers from both fading and additive noises. We allow any form of causal channel encoders/decoders, unlike linear encoders/decoders commonly studied in the literature. Sufficient conditions and necessary conditions for the mean square stabilizability are given in terms of channel parameters such as transmission power and fading and additive noise statistics in relation to the unstable eigenvalues of the open-loop system matrix. The corresponding mean square capacity of the power constrained fading channel under causal encoders/decoders is given. It is proved that this mean square capacity is smaller than the corresponding Shannon channel capacity. In the end, numerical examples are presented, which demonstrate that the causal encoders/decoders render less restrictive stabilizability conditions than those under linear encoders/decoders studied in the existing works.
\end{abstract}

\section{Introduction}

Control over communication networks has been a hot research topic in the past decade \cite{Zaidi2014}. This is mainly motivated by the rapid development of wireless communication technology that enables the connection of geographically distributed systems and devices. However, the insertion of wireless communication networks also poses challenges in analysis and design of control systems due to constraints and uncertainties in communications. One must take the communication networks into consideration and analyze how they affect the stability and performance of the closed-loop control systems.

Until now, there have been plentiful results that reveal requirements on communication channels to ensure the stabilizability. For noiseless digital channels, the celebrated data rate theorem is given in \cite{Nair2004SIAM}. For noisy channels, the problem is complicated by the fact that different channel capacities are required under different stability definitions. For almost sure stability, \cite{Matveev2007} shows that the Shannon capacity in relation to unstable dynamics of a system constitutes the critical condition for its stabilizability. While for moment stability, \cite{Sahai2006} shows that the Shannon capacity is too optimistic while the zero-error capacity is too pessimistic, and the anytime capacity introduced in this paper characterizes the stabilizability conditions. Essentially, to keep the $\eta$-moment of the state of an unstable scalar plant bounded, it is necessary and sufficient for the feedback channel's anytime capacity corresponding to anytime-reliability $\alpha=\eta \mathrm{log}_2|\lambda|$ to be greater than $\mathrm{log}_2 |\lambda|$, where $\lambda$ is the unstable eigenvalue of the plant.  The anytime capacity has a more stringent reliability requirement than  the Shannon capacity. However, it is worthy noting that there exist no systematic method to calculate the anytime capacities of channels.

In control community, the anytime capacity is usually studied under the mean square stability requirement, for which the anytime capacity is commonly named as  the mean square capacity. For example, \cite{Elia2005}  characterizes  the  mean square capacity   of a fading channel.   \cite{Braslavsky2007} studies the  mean square  stabilization  problem over  a power constrained  AWGN channel and characterizes  the critical capacity to  ensure mean square stabilizability.  They   further  show   that  the  extension   from  linear encoders/decoders to more general  causal encoders/decoders cannot provide additional  benefits of  increasing   the  channel   capacity \cite{Freudenberg2010}.
Specifically, the results stated above deal with fading channels or AWGN channels separately. While in wireless communications, it is practical to consider them as a whole.  In this paper, we are interested in a power constrained fading channel which is corrupted by both fading and AWGN. We aim to find the critical condition on the channel to ensure the mean square stabilizability of the system. Note that \cite{Xiao2011} has derived the necessary and sufficient condition for such kind of channel to ensure mean square stabilizability under a linear encoder/decoder. It is still unknown whether we can achieve a higher channel capacity with more general causal strategies. This paper provides a positive answer to this question.

This paper is organized as follows. Problem formulation and some preliminaries  are given in Section 2. Section 3 provides the results for scalar systems. Section 4 discusses the extension to vector systems.  Section 5 provides numerical illustrations and this paper ends with some concluding remarks in Section 6.

\section{Problem Formulation and Preliminaries}

This paper studies the following single-input discrete-time linear system
\begin{equation}
\label{LTIDynamics}
  x_{t+1}=A x_{t}+B u_{t}
\end{equation}
where $x\in \mathbb{R}^n$ is the system state and $u \in \mathbb{R}$ is the control input. Without loss of generality, we assume that all the eigenvalues of $A$ are unstable, i.e.,  $|\lambda_i(A)|\ge 1 $  for all $i=1,2,\ldots, n$ \cite{Freudenberg2010}. The initial value $x_0$ is randomly generated from a Gaussian distribution with zero mean and  bounded covariance ${ \Sigma_{x_0}}$. The system state $x_t$ is observed by a sensor and then encoded and transmitted to the controller through a power constrained fading channel. The communication channel is modeled as
\begin{equation}
\label{channel1}
  r_t=g_ts_t+n_t
\end{equation}
in which $s_t$ denotes the channel input; $r_t$ represents the channel output; $\{g_t\}$ is an i.i.d. stochastic process modeling the fading effects and $\{n_t\}$ is the additive white Gaussian noise with zero-mean  and known variance $\sigma_n^2$. The channel input $s_t$ must satisfy an average power constraint, i.e., $\E\{s_t^2\}\le P$.  We also assume that $x_0, g_0, n_0, g_1, n_1, \ldots$  are independent.  In the paper, it is assumed that after each transmission, the instantaneous value of the fading factor $g_t$ is known to the decoder, which is a reasonable assumption for slowly varying channels with channel estimation \cite{Goldsmith1997}.
The instantaneous Shannon  channel capacity is  $c_t=\frac{1}{2}\mathrm{ln}\big( 1+\frac{g_t^2P}{\sigma_n^2} \big)$ with $c_t$ being measured in nats/transmission. The feedback configuration among the plant, the sensor and the controller, and the channel encoder/decoder structure are depicted in  Fig. 1.
\begin{figure}[htpb]
\centering
  \includegraphics[width=0.4\textwidth]{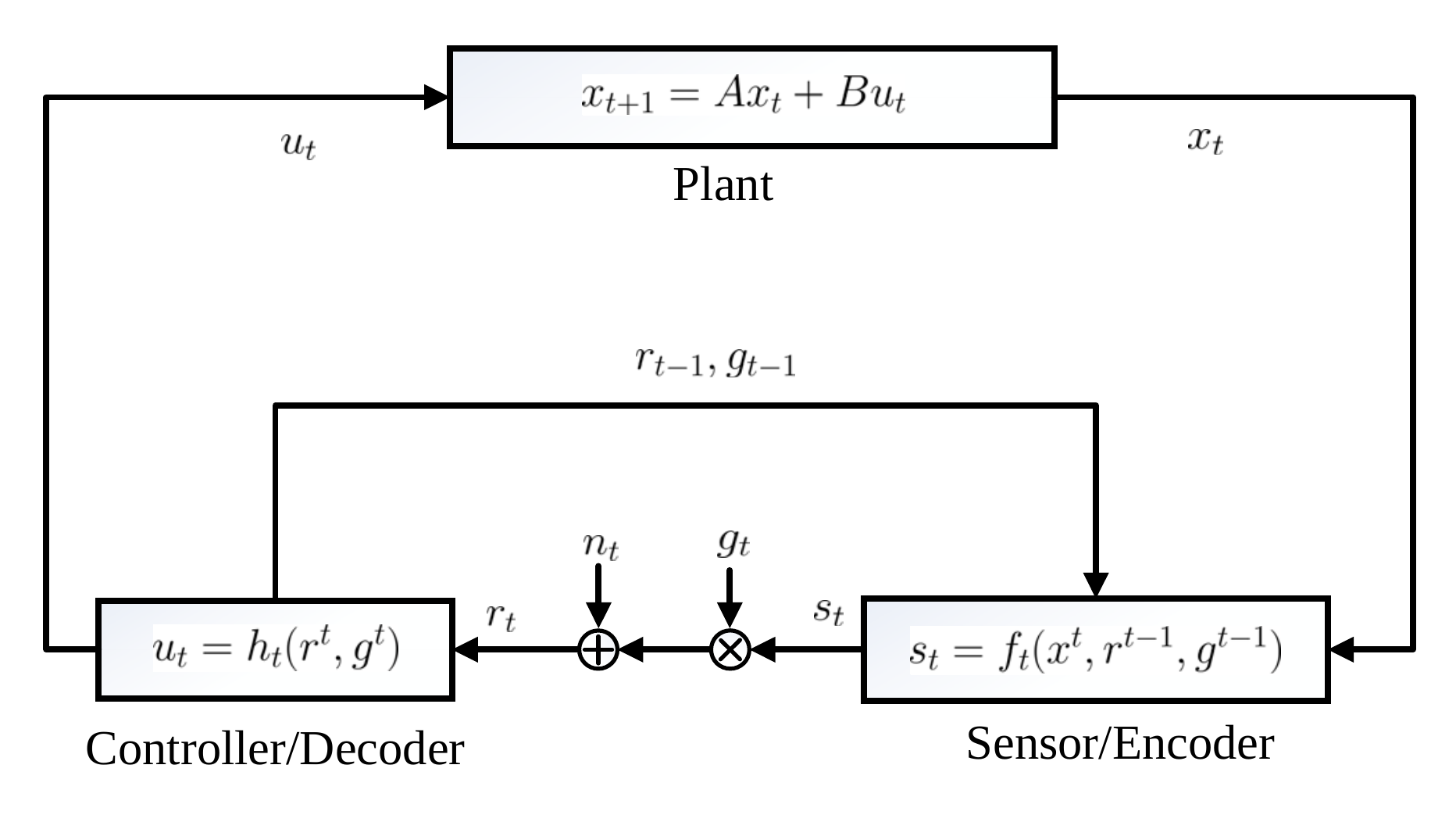}\\
\caption{Network control structure over power constraint fading channel}
\end{figure}
In this paper, we  try to find requirements on the power constrained fading channel such that there exists a pair of causal encoder/decoder $\{f_t\}, \{h_t\}$ that can mean square stabilize the LTI dynamics \eqref{LTIDynamics}, i.e., to render $\mathrm{lim}_{t \rightarrow \infty} \E\{x_tx_t'\}=0$.

To solve this problem, the following preliminaries are needed, which are borrowed from \cite{Freudenberg2010}. Throughout the paper, a sequence $\{\chi_i\}_{i=0}^t$ is denoted by $\chi^t$; random variables are denoted by uppercase letters, and their realizations by lower case letters. All random variables are assumed to exist on a common probability space with measure $\mathcal{P}$. The probability density of a random variable $X$ in Euclidean space with respect to Lebesgue measure on the space is denoted by $p_X$, and the probability density of $X$ conditioned on the $\sigma$-field generated by the event $Y=y$ by $p_{X|y}$. Let the expectation operator be denoted by $\E$, and the expectation conditioned on the event $Y=y$ by $\E_{y}$. We use $\mathrm{log}$ to denote the logarithm to the base two, and $\mathrm{ln}$ to denote the natural logarithm.

The differential  entropy of $X$ is  defined by $H(X)=-\E\{\mathrm{ln} p_X  \}$,  provided  that  the defining  integral  exists.  Denote  the conditional   entropy    of   $X$    given   the   event    $Y=y$   by $H_y(X)=H(X|Y=y)=-\E_y\{  \mathrm{ln}  p_{X|y}  \}$,  and  the  random variable associated with $H_y(X)$ by $H_Y(X)$. The average conditional entropy of $X$ given the event $Y=y$  and averaged over $Y$ is defined by $H(X|Y)=\E\{H_Y(X) \}$,  and the average conditional entropy  of $X$ given the  events   $Y=y$  and   $Z=z$  and  averaged   only  over   $Y$  by $H_z(X|Y)=\E_{z}\{H_{Y,Z}(X)\}$.    The conditional mutual information between two random variables $X$ and $Y$ given the event $Z=z$ is defined by $I_z(X;Y)=H_z(X)-H_z(X|Y)$. Given  a   random  variable   $X\in \mathbb{R}^{n}$  with entropy  $H(X)$,  the entropy  power  of $X$  is defined  by $N(X)=\frac{1}{2\pi  e}  e^{\frac{2}{n}H(X)}$. Denote  the conditional  entropy   power  of   $X$  given   the  event   $Y=y$  by $N_y(X)=\frac{1}{2\pi   e}e^{\frac{2}{n}H_y(X)}$,   and   the   random variable associated with $N_y(X)$ by $N_Y(X)$. The average conditional entropy power  of $X$ given the  event $Y=y$ and averaged  over $Y$ is defined by $N(X|Y)=\E\{N_Y(X)\}$, and  the average conditional entropy power of $X$  given the events $Y=y$ and $Z=z$  and averaged only over $Y$ by $N_z(X|Y)=\E_z \{N_{Y,Z}(X)\}$.  The following lemma shows that the entropy power of a random  variable provides an estimation of the lower bound for its variance.

\begin{lemma}[\cite{Freudenberg2010}]
\label{lemma:varianceIsBoundedByEntropyPower}
Let $X$ be an $n$-dimensional random variable. Then $N_y(X)\le \frac{1}{n} \E_y\{\|X\|^2\}$.
\end{lemma}

\begin{lemma}
\label{lemma:mutualInformationEqual}
 Let $X$ be an $n$-dimensional random variable, $f(X)$ be a function of $X$, and $Y=f(X)+N$ with $N$ being a random variable that is independent with $X$. Then  $I(X;Y)=I(f(X);Y)$.
\end{lemma}
\begin{proof}
Since $H(Y|X)=H(Y|X,f(X))\le H(Y|f(X))$, we have $H(Y)=I(X;Y)+H(Y|X)\le I(X; Y)+H(Y|f(X))$. Thus $ H(Y)-H(Y|f(X)) = I(Y;f(X))\le I(X;Y)$.
Besides, noting that $X\rightarrow f(X)\rightarrow Y$ forms a Markov chain, the data processing inequality \cite{Cover2006} implies that $I(X;Y) \le I(f(X); Y)$. Combining the two facts, we have  $I(X;Y)=I(f(X);Y)$.
\end{proof}
\begin{remark}
  Lemma \ref{lemma:mutualInformationEqual}  indicates that for the AWGN channel,  the amount of information that the channel output contains about the source is equal to the amount of information that the channel output contains about the channel input.
\end{remark}

\section{Scalar Systems}
To better convey our ideas, we start with scalar systems. Consider the following scalar system
\begin{equation}
\label{scalarDynamics}
  x_{t+1}=\lambda x_t+u_t
\end{equation}
where $|\lambda|\ge 1$ and $\E\{x_0^2\}=\sigma_{x_0}^2$.
With the communication channel given in \eqref{channel1}, the stabilizability result is stated in the following theorem.
\begin{theorem}
\label{theorem:theorem1}
There exists a causal encoder/decoder pair $\{f_t\}, \{h_t\}$, such that the system \eqref{scalarDynamics} can be stabilized over the communication channel \eqref{channel1} in mean square sense if and only if
\begin{equation}
\label{iffConditionForScalarSystem}
 \mathrm{log} |\lambda| < - \frac{1}{2} \mathrm{log}  \E\{ \frac{\sigma_n^2}{\sigma_n^2+g_t^2P} \}
  \end{equation}
\end{theorem}
Theorem \ref{theorem:theorem1} indicates that the mean square capacity of the power constraint fading channel is $C_{\mathrm{MSC}}=-\frac{1}{2} \mathrm{log} {\E\{ \frac{\sigma_n^2}{\sigma_n^2+g_t^2P} \}}$. In the following, we will prove the necessity and sufficiency of Theorem \ref{theorem:theorem1}, respectively. The proof essentially follows the same steps as in \cite{Minero2009, Freudenberg2010, Kumar2014}, however, with some differences due to the channel structure.

\subsection{Proof of Necessity}
The proof of necessity follows from the intuition below.  In view of Lemma \ref{lemma:varianceIsBoundedByEntropyPower}, the entropy power provides a lower bound for the mean square value of the system state. We thus can use the average entropy power as a measure of the uncertain region of the system state and analyze its update. At time $t$, the controller maintains a knowledge of the uncertain region of $x_t$. When it takes action on the plant, the average uncertain region  of $x_{t+1}$ predicated by the controller is expanded to $\lambda^2$ times that of $x_t$. This is the iteration we term as dynamics update, which describes the update of the uncertain region of $x$ maintained by the controller from time $t$ to $t+1$.  After receiving  information about $x_{t+1}$ from the sensor through the communication channel,  the controller can reduce the predication error of the uncertain region of  $x_{t+1}$ by a factor of $\E\{\frac{\sigma_n^2}{\sigma_n^2+g_t^2P}\}$. This is the iteration we term as communication update, which describes the update of the uncertain region of $x$ maintained by the controller at time $t+1$ after it has received the information about $x_{t+1}$ from the sensor through the communication channel.   Thus to ensure mean square stability, the average expanding factor $\lambda^2 \E\{\frac{\sigma_n^2}{\sigma_n^2+g_t^2P} \}$ of the system state's uncertain region should be smaller than one, which gives the necessary requirement in Theorem \ref{theorem:theorem1}. The formal proof is stated as follows. Here we use the uppercase letters $X, S, R, G$ to denote the random variables of the system state, the channel input, the channel output and the channel fading coefficient. We use the lowercase letters $x, s, r, g$ to denote their realizations.
      \subsubsection{Communication Update}
     The average  entropy power of $X_t$ conditioned on  $(R^t,G^t)$ is
$\scriptstyle N(X_t|R^t,G^t)=  \E\{ N_{R^t,G^t}(X_t) \} \overset{(a)}{=} \E \{ \E \{N_{R^t,G^t}(X_t) | R^{t-1}, G^t \} \}
         \overset{(b)}{=}\frac{1}{2\pi e} \E \{ \E \{ e^{ {2} H_{R^t,G^t}(X_t)} | R^{t-1}, G^t  \}\} $
    where $(a)$ follows from the law of total expectation and $(b)$ follows from the definition of entropy power.
  Since
   $
   \begin{aligned}
   &        \E \{ e^{2 H_{R^t,G^t} (X_t) }   | R^{t-1}= r^{t-1}, G^t= g^t \} \\
     &       \overset{(c)}{\ge} e^{ 2 \E  \{ H_{R^t,G^t} (X_t) | R^{t-1}= r^{t-1}, G^t= g^t   \} }\\
     &\overset{(d)}{=} e^{2  H(X_t| R_t,  R^{t-1}= r^{t-1}, G^t= g^t) }\\
     &\overset{(e)}{=} e^{2 \left( H(X_t| R^{t-1}= r^{t-1}, G^t= g^t) - I(X_t, R_t| R^{t-1}= r^{t-1}, G^t= g^t) \right)}\\
     &\overset{(f)}{=} e^{2  \left( H(X_t| R^{t-1}= r^{t-1}, G^t= g^t) - I(S_t, R_t| R^{t-1}= r^{t-1}, G^t= g^t) \right) }\\
    & \overset{(g)}{\ge} e^{2  \left( H(X_t | R^{t-1}= r^{t-1}, G^t= g^t) - c_t \right) }\\
     & \overset{(h)}{=} e^{- 2 c_t} e^{2 H(X_t|R^{t-1}=r^{t-1},G^{t-1}=g^{t-1})}
    \end{aligned}
     $
    where  $(c)$ follows from Jensen's inequality; $(d)$ follows from the definition of conditional entropy;  $(e)$ follows from the definition of conditional  mutual information; $(f)$ follows from Lemma \ref{lemma:mutualInformationEqual}; $(g)$ follows from the definition of channel capacity, i.e.,  $I(S_t, R_t|R^{t-1}=r^{t-1}, G^t=g^t )\le c_t$ and  $(h)$ follows from the fact that $G_t$ is independent with $X_t$,  we have
    $\scriptstyle
            N(X_t|R^t,G^t)
         \ge\frac{1}{2\pi e} \E \{  e^{-2 C_t} e^{2 H_{R^{t-1},G^{t-1}}(X_t)} \}
         =\E \{  \frac{\sigma_n^2}{\sigma_n^2+g_t^2P} \}  N(X_t|R^{t-1},G^{t-1})
       $.

    \subsubsection{Dynamics Update}
    Since
$e^{2H(X_{t+1}|R^t=r^t,G^t=g^t)}  = e^{2 H(\lambda X_t+U_t|R^t=r^t,G^t=g^t)} \overset{(i)}{=} e^{2 H (\lambda X_t|R^t=r^t,G^t=g^t)}
\overset{(j)}{=} e^{2H(X_t|R^t=r^t,G^t=g^t)+2\ln|\lambda|}
 = \lambda^2  e^{2 H(X_t|R^t=r^t,G^t=g^t)}
$
where $(i)$ follows from the fact that $u_t=h_t(r^t, g^t)$ and $(j)$ follows from
 Theorem 8.6.4 in \cite{Cover2006}, we have
    $ \scriptstyle N(X_{t+1}|R^t,G^t)   \ge \E \left\{\frac{1}{2\pi e} \lambda^2  e^{2 H_{R^t,G^t}(X_t)}  \right\}    = \lambda^2 N(X_t|R^t,G^t)  $.
    \subsubsection{Proof of Necessity}
    Combining the results of communication update and dynamics update, we have
$   \scriptstyle  N(X_{t+1}|R^t,G^t)  \ge   \lambda^2 \E \{ \frac{\sigma_n^2}{\sigma_n^2+g_t^2P} \} N(X_t|R^{t-1}, G^{t-1})$.
    In view of Lemma \ref{lemma:varianceIsBoundedByEntropyPower}, $N(X_{t+1}|R^t,G^t)$ should converge to zero asymptotically. Thus $\lambda^2 \E \{ \frac{\sigma_n^2}{\sigma_n^2+g_t^2P} \} <1$, which is  \eqref{iffConditionForScalarSystem}  and this proves the necessity.

\subsection{Proof of Sufficiency}
To prove the sufficiency, we need to construct a pair of encoder and decoder. The encoder and decoder are designed following an "estimation then control" strategy. The controller consecutively  estimates the initial state $x_0$ by using the received information from the channel and then applies an equivalent control to the plant. The reason for adopting such strategy is explained as follows. The response of the linear system is $x_t=\lambda^t (x_0-\hat{x}_t)$ with $\hat{x}_t=-\sum_{i=0}^{t-1} \lambda^{-1-i} u_i$, which means $\E\{x_t^2\}=\lambda^{2t} \E\{(x_0-\hat{x}_t)^2\}$. We can treat $\hat{x}_t$ as an estimate of the controller for the initial state $x_0$. If the estimation error $\E\{(x_0-\hat{x}_t)^2\}$ converges to zero at a speed that is greater than $\lambda^2$, i.e., there exists $\eta>\lambda^2$ and $\alpha>0$, such that $\E\{(x_0-\hat{x}_t)^2\}\le \frac{\alpha}{\eta^t}$, the mean square value of the system state would be bounded by
$
  \E\{x_t^2\}\le \alpha \left(\frac{\lambda^2}{\eta}\right)^{t}
$.
Thus $\underset{t\rightarrow \infty}{\mathrm{lim}}\E\{x_t^2\}=0$, i.e.,  system \eqref{scalarDynamics} is mean square stable. This intuition can be formalized using the following lemma.
\begin{lemma}[\cite{Kumar2014}]
\label{lemma:estimationThenControl}
If there exists an estimation scheme $\hat{x}_t$ for the initial system state $x_0$, such that the estimation error $e_t=\hat{x}_t-x_0$ satisfies the following property,
\begin{eqnarray}
\label{eq:estimationThenControlRequirement}
\E\{ e_t \}=0 \label{eq:estimationThenControlRequirement1} \\
\lim_{t\rightarrow \infty} A^t \E\{ e_te_t' \} (A')^t=0  \label{eq:estimationThenControlRequirement2}
\end{eqnarray}
then the system \eqref{LTIDynamics} can be mean square stabilized by the  controller
$   u_t=K\left( A^t \hat{x}_t+ \sum_{i=1}^t A^{t-i} Bu_{i-1} \right) $
with $K$ being selected such that $A+BK$ is stable.
\end{lemma}

 When $g_t$ is known at the receiver, channel \eqref{channel1} resembles an AWGN channel. Shannon shows that  when estimating a Gaussian random variable through an AWGN channel, the minimal mean square estimation error can be attained by using linear encoders and decoders, respectively \cite{Gattami2014}. And the minimal mean square error variance is given by $\frac{P\sigma_n^2}{\sigma_n^2+g_t^2P}$.  Thus through one channel use, we can at best decrease the estimation error by a factor of $\frac{\sigma_n^2}{\sigma_n^2+g_t^2P}$. Since ${g_t}$ is i.i.d., we can transmit the estimation error from the decoder to the encoder and iteratively conduct the minimal mean square estimation process. Then the estimation error would decrease on average at a speed of $\E \{ \frac{\sigma_n^2}{\sigma_n^2+g_t^2P}\}$. If $\lambda^2\E\{\frac{\sigma_n^2}{\sigma_n^2+g_t^2P}\}<1$, in view of Lemma \ref{lemma:estimationThenControl}, system \eqref{scalarDynamics} can be mean square stabilized.  The estimation strategy actually follows the principle of the well-known scheme of Schalkwijk \cite{Schalkwijk1966}, which  utilizes the noiseless feedback link to consecutively refine the estimation error. The detailed encoder/decoder design and stability analysis are given as follows.

    \subsubsection{Encoder/Decoder Design}
Suppose the estimation of $x_0$ formed by the decoder is $\hat{x}_t$ at time $t$  and the estimation error is $e_t=\hat{x}_t-x_0$. The encoder is designed as
\begin{equation}
\label{encoder}
  \begin{aligned}
     s_0  &=\sqrt{\frac{P}{\sigma_{x_0}^2}} x_0\\
      s_t&=\sqrt{\frac{P}{\sigma^2_{e_{t-1}}}}\left( \hat{x}_{t-1} - x_0 \right), \;\; t\ge 1\\
  \end{aligned}
\end{equation}
The decoder is designed as
\begin{equation}
\label{decoder}
  \begin{aligned}
    \hat{x}_0 & =\sqrt{\frac{\sigma_{x_0}^2}{P}}r_0\\
     \hat{x}_t&=\hat{x}_{t-1}-\frac{\E\{r_t e_{t-1}|g_t\}}{\E\{r_t^2|g_t\}}r_t, \;\; t\ge 1
  \end{aligned}
\end{equation}
with $\sigma^2_{e_{t-1}}$ representing the variance of $e_{t-1}$.

    \subsubsection{Proof of Sufficiency}
Since $r_0=g_0s_0+n_0$, in view of \eqref{encoder} and \eqref{decoder}, we have $  e_0 =(g_0-1)x_0+ \sqrt{\frac{\sigma_{x_0}^2}{P}}n_0
$.  Because $g_0$, $x_0$, $n_0$ are independent and $x_0$, $n_0$ follows a zero mean Gaussian distribution, we know that the conditional probability distribution of $e_0$ given the event  $g_0$ is Gaussian and $  \E\{e_0|g_0\}  =0$,   $\E\{e_0^2|g_0\}  = (g_0-1)^2 \sigma_{x_0}^2+ \frac{\sigma_{x_0}^2\sigma_n^2}{P}$. Thus $\E\{e_0\}=\E \{\E\{e_0|g_0\} \} =0$ and  $ \E\{e^2_0\} =\E\{ \E\{e_0^2|g_0\} \} = \E \{(g_0-1)^2\} \sigma_{x_0}^2+ \frac{\sigma_{x_0}^2\sigma_n^2}{P}$.

For $t\ge 1$, in view of \eqref{encoder} and  \eqref{decoder}, we have
\begin{multline*}
e_t=e_{t-1}-\frac{\E\{r_te_{t-1}|g_t\}}{\E\{r_t^2|g_t\}}r_t\\
    =\Big(1-g_t \sqrt{\frac{P}{\sigma_{e_{t-1}}^2}} \frac{\E\{r_te_{t-1}|g_t\}}{\E\{r_t^2|g_t\}} \Big) e_{t-1} -\frac{\E\{ r_t e_{t-1}|g_t \}}{\E\{r_t^2|g_t\}} n_t
    \end{multline*}
Thus the conditional probability distribution for $e_t$ given the event $g_t$ is Gaussian.
We also have
\begin{equation*}
\begin{aligned}
& \E\{e_t\}  =  \E\{ \E\{ e_t|g_t \} \}\\
&=  \E\Big\{ \Big(1-g_t \sqrt{\frac{P}{\sigma_{e_{t-1}}^2}} \frac{\E\{r_te_{t-1}|g_t\}}{\E\{r_t^2|g_t\}} \Big) \E\{e_{t-1}|g_t\} \Big\} \\
&\overset{(a)}{=} \E \Big\{ \Big(1-g_t \sqrt{\frac{P}{\sigma_{e_{t-1}}^2}} \frac{\E\{r_te_{t-1}|g_t\}}{\E\{r_t^2|g_t\}} \Big) \Big\} \E\{e_{t-1}\}
\end{aligned}
\end{equation*}
    where $(a)$ follows from the fact that $g_t$ is independent with $e_{t-1}$. Since $\E\{e_0\}=0$, we further know that  $\E\{e_t\}\equiv 0$. The sufficient condition \eqref{eq:estimationThenControlRequirement1} is satisfied.

Since $e_{t-1}$, $g_t$ and $n_t$ are  independent, we have $\E\{e_{t-1}^2|g_t\}=\E\{e_{t-1}^2\}$ and $\E\{n_t^2|g_t\}=\E\{n_t^2\}$, which implies
$       \E\{r_t^2|g_t\}    = \E\big\{ \big(g_t \sqrt{\frac{P}{\sigma^2_{e_{t-1}}}} e_{t-1} +n_t\big)^2|g_t \big\}
                                     =\sigma_n^2+g_t^2P
$
 and
$
    \E\{r_te_{t-1}|g_t\}     =\E\big\{e_{t-1}\big(g_t \sqrt{\frac{P}{\sigma^2_{e_{t-1}}}} e_{t-1} +n_t \big)|g_t \big\}
                                         = g_t \sqrt{P \sigma_{e_{t-1}}^2}
$.
Since $ \E\{e_t^2|g_t\}=\E\{e_{t-1}^2|g_t\} -\frac{\E\{r_te_{t-1}|g_t\}^2}{\E\{r_t^2|g_t\}} $, we also have
$
     \E\{e_t^2|g_t\}  =\E\{e_{t-1}^2\}-\frac{g_t^2P \E\{e_{t-1}^2\}}{\sigma_n^2 + g_t^2P}
                                 =\E\{e_{t-1}^2\} \frac{\sigma_n^2}{\sigma_n^2 + g_t^2P}
$,   which implies
$       \E\{e_t^2\}  = \E\{\E\{e_t^2|g_t\}\}
                             =\E\{e_{t-1}^2\} \E\{ \frac{\sigma_n^2}{\sigma_n^2 + g_t^2P}\}
$.
   Thus if $\lambda^2  \E\{ \frac{\sigma_n^2}{\sigma_n^2 + g_t^2P}\}<1$, the designed encoder/decoder pair can guarantee  \eqref{eq:estimationThenControlRequirement2}. In view of Lemma \ref{lemma:estimationThenControl}, the sufficiency of Theorem \ref{theorem:theorem1} is proved.

\begin{remark}
We can show that $C_{\mathrm{MSC}}$ is smaller than the Shannon capacity, which is $C_{\mathrm{Shannon}}=\E\{c_t\}$ \cite{Goldsmith1997}. From Jensen's inequality, we know that $\E\{2^{-2c_t}\}\ge 2^{-2\E\{c_t\}}$ and the equality holds if and only if $c_t$ is a constant. Thus it follows that $  C_{\mathrm{MSC}} = \frac{1}{2}\mathrm{log}\frac{1}{\E\{2^{-2c_t}\}}\le \frac{1}{2}\mathrm{log} \frac{1}{2^{-2\E\{c_t\}}}=\E\{c_t\}=C_{\mathrm{Shannon}} $ and the equality holds if and only if $c_t$ is a constant.
\end{remark}

\begin{remark}
By letting $g_t$ in~\eqref{iffConditionForScalarSystem}  be the Bernoulli distribution with failure probability $\epsilon$, and taking the limit $\sigma_n^2 \rightarrow 0$ and $P \rightarrow \infty$, we can show that the necessary and sufficient condition to ensure mean square stabilizability for the real erasure channel is $\epsilon < \frac{1}{\lambda^2}$, which recovers the result in \cite{Elia2005}.  If we let $g_t$  be a constant with $g_t=1$, then the studied power constrained fading channel degenerates to the AWGN channel and the \eqref{iffConditionForScalarSystem} degenerates to $\frac{1}{2} \mathrm{log} (1+\frac{P}{\sigma_n^2})< \mathrm{log} |\lambda|$, which recovers the result in \cite{Sahai2006, Braslavsky2007}. If $\sigma_n^2=0$ and the event $g_t=0$ has zero probability measure, the right hand side of \eqref{iffConditionForScalarSystem} becomes infinity. Then for any $\lambda$, \eqref{iffConditionForScalarSystem} holds automatically. This is reasonable since we have assumed that $g_t$ is known at the decoder side, thus if there is no additive noise, the channel resembles a perfect communication link. Since \eqref{scalarDynamics} is controllable, we can always find a pair of encoder and decoder to  stabilize the system.
\end{remark}

\section{Vector Systems}

For vector systems, the situation becomes complicated by the fact that we have $n$ sources $x_{i,0}$ and only one channel, where $x_{i,0}$ denotes the $i$-th element of $x_0$.  Firstly, we would analyze  the achievable minimal mean square estimation error for estimating $x_0$ over the channel \eqref{channel1} during one channel use. Consider the following Markov chain
\begin{equation*}
  X_0\rightarrow S_t= f_t(X_0)\rightarrow R_t\rightarrow \hat{X}_t=h_t(R_t)
\end{equation*}
where $X_0\in \mathbb{R}^n$ denotes the Gaussian initial state with covariance matrix $\Sigma_{x_0}$; $f_t(\cdot)$ is a scalar-valued function denoting the channel encoder for \eqref{channel1}; $R_t$ denotes the channel output and $\hat{X}_t$ is the estimation of $X_0$ formed by the decoder with decoding rule $h_t(\cdot)$.

Denote the estimation error as $e_t=X_0-\hat{X}_t$, in view of Lemma \ref{lemma:varianceIsBoundedByEntropyPower}, we have
$
  \frac{1}{n} \mathrm{tr} \E \{ e_te_t' \} \ge \frac{1}{2\pi e} e^{\frac{2}{n} H(e_t|R_t)}
$.
Since
\begin{equation*}
\begin{aligned}
  H(e_t|R_t) &= H(X_0-h_t(R_t)|R_t)=H(X_0|R_t)\\
                    &=H(X_0)-I(X_0;R_t)\\
                    &\overset{(a)}=H(X_0)-I(f_t(X_0);R_t)\\
                    &\ge \frac{1}{2} \mathrm{ln} ((2 \pi e)^n \mathrm{det}(\Sigma_{x_0}))-\frac{1}{2} \mathrm{ln}(1+\frac{g_t^2P}{\sigma_n^2})\\
  \end{aligned}
\end{equation*}
where $(a)$ follows from Lemma \ref{lemma:mutualInformationEqual}, thus we have
\begin{equation*}
   \mathrm{tr} \E \{ e_te_t' \} \ge n \; \mathrm{det} (\Sigma_{x_0}) \big( \frac{\sigma_n^2}{g_t^2P+\sigma_n^2} \big)^{\frac{1}{n}}
\end{equation*}
From the above inequality, we know that the minimal mean square error is  given in terms of  $\frac{\sigma_n^2}{g_t^2P+\sigma_n^2} $. However, this is only for the sum of the estimation errors $e_{i,t}$ with $e_{i,t}$ being the $i$-th element of $e_t$. There is no indication on the convergence speed for every single $e_{i,t}$. Lemma \ref{lemma:estimationThenControl} implies that we should design the encoder/decoder to render that $\mathrm{lim}_{t \rightarrow \infty} \lambda_i^{2t} \E\{e_{i,t}^2\}=0$ for all $i$, which places separate requirements for the convergence speed of each $e_{i,t}$. Thus we need to optimally allocate channel resources to each unstable state variable.

The previous analysis also implies that we should treat the unstable modes of $A$ separately. Here we  focus on the real Jordan canonical form of system \eqref{LTIDynamics}. Let $\lambda_1, \ldots, \lambda_d$ be the distinct unstable eigenvalues (if $\lambda_i$ is complex, we exclude from this list the complex conjugates $\lambda_i^*$) of $A$ in \eqref{LTIDynamics}, and let $m_i$ be the algebraic multiplicity of each $\lambda_i$. The real Jordan canonical form $J$ of $A$ then has the block diagonal structure $J=\mathrm{diag}(J_1,\ldots,J_d)\in \mathbb{R}^{n\times n}$, where the block $J_i\in \mathbb{R}^{\mu_i\times \mu_i}$ and $\mathrm{det} J_i =\lambda_i^{\mu_i}$, with
\begin{equation*}
  \mu_i =\left\{ \begin{matrix}
    m_i & \mathrm{if} \;\; \lambda_i \in \mathbb{R}\\
    2m_i & \mathrm{otherwise}
  \end{matrix} \right.
\end{equation*}
It is clear that we can equivalently study the following dynamical system instead of \eqref{LTIDynamics}
\begin{equation}
\label{realJordanCanonicalForm}
  x_{k+1}=Jx_k+TBu_i
\end{equation}
for some similarity matrix $T$. Let $\mathcal{U}=\{1,\ldots, d\}$ denote the index set of unstable eigenvalues.
\begin{theorem}
  \label{theorem:VectorResult}
  There exists a causal encoder/decoder pair $\{f_t\}, \{h_t\}$, such that the LTI dynamics \eqref{LTIDynamics} can be stabilized over the communication channel \eqref{channel1} in mean square sense if
\begin{equation}
\label{SufficientConditionForVectorSystem}
 \sum _{i=1}^d \mu_i\mathrm{log} |\lambda_i| < - \frac{1}{2} \mathrm{log} { \E\{ \frac{\sigma_n^2}{\sigma_n^2+g_t^2P} \}}
  \end{equation}
  and only if $(\mathrm{log}|\lambda_1|, \ldots, \mathrm{log}|\lambda_d|) \in \mathbb{R}^{d}$ satisfy that for all $v_i \in \{0, \ldots, m_i\}$ and $i\in \mathcal{U}$
  \begin{equation}
  \label{necessityForVectorSystem}
    \sum_{i \in \mathcal{U}} a_i v_i \mathrm{log}|\lambda_i|< - \frac{v}{2} \mathrm{log} {\E \big\{ \big(\frac{\sigma_n^2}{\sigma_n^2+g_t^2P}\big)^{\frac{1}{v}}\big\}}
  \end{equation}
  where $v=\sum_{i\in \mathcal{U}} a_iv_i$, and $a_i=1$ if $\lambda_i \in \mathbb{R}$, and $a_i=2$ otherwise.
\end{theorem}

\begin{proof}
For the proof of necessity, notice that each block $J_i$ has an invariant real subspace $\mathcal{A}_{v_i}$ of dimension $a_i v_i$, for any $v_i \in \{ 0 , \ldots,m_i\}$. Consider the subspace $\mathcal{A}$ formed by taking the product of the invariant subspaces $\mathcal{A}_{v_i}$ for each real Jordan block. The total dimension of $\mathcal{A}$ is $v=\sum_{i\in \mathcal{U}} a_iv_i$. Denote by $x^\mathcal{V}$ of the components of $x$ belonging to $\mathcal{A}$.  Then $x^{\mathcal{V}}$ evolves as
\begin{equation}
\label{stackedDynamics}
  x_{k+1}^{\mathcal{V}}=J^{\mathcal{V} }x_{k+1}^{\mathcal{V}} +QT u_k
\end{equation}
where $Q$ is a transformation matrix and $\mathrm{det} J^{\mathcal{V}}=\Pi_{i\in \mathcal{U}} \lambda_i^{ a_i v_i}$.
 Since $X_k$ is mean square stable, it is necessary that the subdynamics \eqref{stackedDynamics} is mean square stable. Similar to the necessity proof in Theorem 1, we may derive the necessary condition  \eqref{necessityForVectorSystem}.  And this completes the proof of necessity.

 Here we prove the sufficiency using the idea of Time Division Multiple Access (TDMA).
 Based on the previous encoder/decoder design for scalar systems, the following information transmission strategy is designed for the vector system. Without loss of generality, here we assume that $\lambda_1,\ldots, \lambda_d$ are real and $m_i=1$. For other cases, readers can refer to the analysis discussed in Chapter 2 of \cite{Zaidi2014}. Specifically, under this assumption, $J$ is a diagonal matrix and $d=n$. The sensor transmits periodically with a period of $\tau$. During one channel use, the sensor only transmits the estimation error of the $j$-th value of $x_0$ using the scheme devised for scalar systems. The relative transmission frequency  for the $j$-th value of $x_0$ is scheduled to be $\alpha_j$ among the $\tau$ transmission period with $\sum_{j=1}^n\alpha_j= 1$. The receiver maintains an array that represents the most recent estimation of $x_0$, which is set to $0$  for $t=0$. When the information about the $j$-th value of $x_0$ is transmitted, only the estimation of the $j$-th value of $x_0$ is updated at the decoder side, and the other estimation values remain unchanged. After updating the estimation, the controller takes action as the one designed in Lemma \ref{lemma:estimationThenControl}.
If the diagonal elements of $ A^t \E\{e_te_t'\}(A')^t$ converge to zeros asymptotically, i.e., for $i=1,\ldots, n$, $\mathrm{lim}_{t\rightarrow \infty} \lambda_i^{2t}  \E\{e_{i,t}^2\}=0$ , the conditions in Lemma \ref{lemma:estimationThenControl} can be satisfied. Since the transmission is scheduled periodically, we only need to require that $\mathrm{lim}_{k\rightarrow \infty} \lambda_i^{2k\tau}
\E\{e_{i, k\tau}^2\}=0$, $\forall i=1,\ldots, n$. Following our designed transmission scheme, we have $\E\{e_{i,k\tau}^2\}=\E\{\frac{\sigma_n^2}{\sigma_n^2+g_t^2P}\}^{\alpha_i k\tau}\E\{e_{i,0}^2\}$. If $\lambda_i ^{2}\E\{\frac{\sigma_n^2}{\sigma_n^2+g_t^2P}\}^{\alpha_i }<1 $ for all $i=1,\ldots n$,  the sufficient condition in Lemma \ref{lemma:estimationThenControl} can be satisfied. To complete the proof, we only need to show the equivalence between the requirement  $\lambda_i ^{2}\E\{\frac{\sigma_n^2}{\sigma_n^2+g_t^2P}\}^{\alpha_i }<1 $ for all $i=1,\ldots n$ and \eqref{SufficientConditionForVectorSystem}. On one hand, since $\sum_{i=1}^n \alpha_i=1$, if $\lambda_i ^{2}\E\{\frac{\sigma_n^2}{\sigma_n^2+g_t^2P}\}^{\alpha_i }<1 $ for all $i=1,\ldots n$, we know that \eqref{SufficientConditionForVectorSystem} holds. On the other hand, if \eqref{SufficientConditionForVectorSystem} holds, we can simply choose  $\alpha_i=\frac{\mathrm{log}|\lambda_i|}{\sum_i \mathrm{log}|\lambda_i|}$, which satisfies the requirement that $\sum_{i=1}^n \alpha_i=1$ and $\lambda_i ^{2}\E\{\frac{\sigma_n^2}{\sigma_n^2+g_t^2P}\}^{\alpha_i }<1 $ for all $i=1,\ldots, n$. The sufficiency is proved. \end{proof}

\section{Numerical Illustrations}
\subsection{Scalar Systems}
The authors in \cite{Xiao2011}  derive the mean square capacity of a power constrained fading channel with linear encoders/decoders. The necessary and sufficient condition  for scalar systems is  $ \frac{1}{2} \mathrm{log}(1+\frac{\mu^2_gP}{\sigma_g^2 P+\sigma_n^2}) >  \mathrm{log} |\lambda|  $ with $\mu_g$ and $\sigma_g^2$  being the mean and variance of $g_t$. We can similarly define the mean square capacity of the power constrained fading channel with linear encoders/decoders as $C_{\mathrm{MSL}}=\frac{1}{2} \mathrm{log}(1+\frac{\mu^2_gP}{\sigma_g^2P+\sigma_n^2})$.  Simply assume that the fading follows the Bernoulli distribution with failure probability $\epsilon$, then the Shannon capacity, the mean square capacity achievable with causal encoders/decoders and the mean square capacity achievable with linear encoders/decoders are given as
$ C_{\mathrm{Shannon_{BD}}}  = \frac{1-\epsilon}{2} \mathrm{log}\big(1+\frac{P}{\sigma_n^2}\big) $,
$ C_{\mathrm{MSC_{BD}}}  =-\frac{1}{2} \mathrm{log} \big(\frac{\sigma_n^2+\epsilon P}{\sigma_n^2+P}\big) $,
 $C_{\mathrm{MSL_{BD}}}  =  \frac{1}{2} \mathrm{log} \big(1+\frac{(1-\epsilon)^2P}{(1-\epsilon) \epsilon P+\sigma_n^2}\big) $.
For fixed $P$ and $\sigma_n^2$, the channel capacities are functions of $\epsilon$. Let $P=1$ and $\sigma_n^2=1$, the channel capacities in relation to the erasure probability are plotted in  Fig.  \ref{comparisonOfChannelCapacity}. It is clear that  $C_{\mathrm{Shannon_{BD}}} \ge C_{\mathrm{MSC_{BD}}}\ge C_{\mathrm{MSL_{BD}}}$ at any given erasure probability $\epsilon$. This result is obvious since we have proved that the Shannon capacity is no smaller than the mean square capacity with causal encoders/decoders. Besides, we have more freedom in designing the causal encoders/decoders compared with the linear encoders/decoders, thus allowing to achieve a higher capacity.  The three kinds of capacity degenerate to the same when $\epsilon=0$ and $\epsilon=1$, which represent the AWGN channel case and the disconnected case respectively.
\begin{figure}
  \centering
  \includegraphics[width=0.37\textwidth]{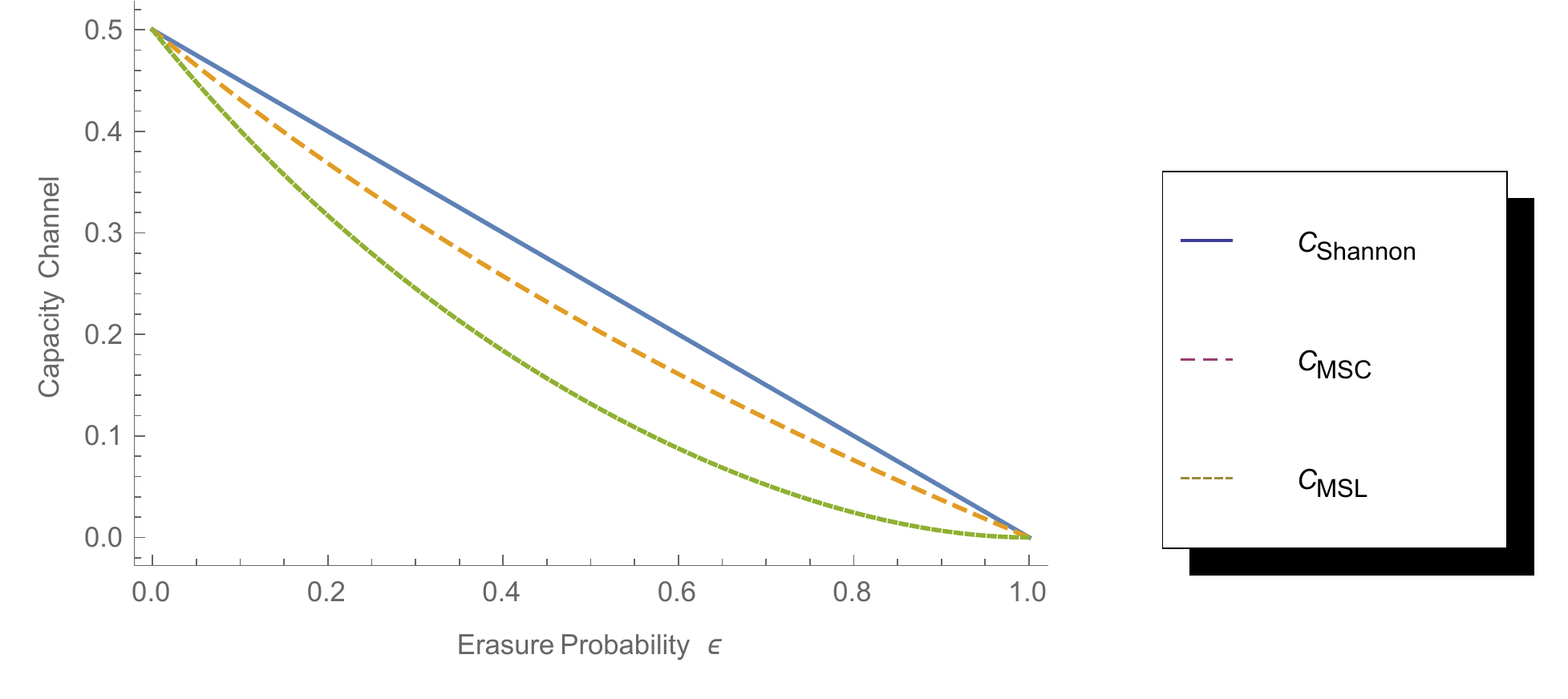}\\
  \caption{Comparison of different channel capacities when $P=1$, $\sigma_n^2=1$}\label{comparisonOfChannelCapacity}
\end{figure}

\subsection{Vector Systems}
Consider the two dimensional LTI system \eqref{realJordanCanonicalForm} with
$
J=\left[
        \begin{smallmatrix}
          \lambda_1 & 0\\
          0     & \lambda_2
        \end{smallmatrix}
    \right]
$, and the communication channel is \eqref{channel1} in which the fading follows the Bernoulli distribution with failure probability $\epsilon$. In view of Theorem \ref{theorem:VectorResult}, a sufficient condition to ensure mean square stabilizability is that $(\mathrm{log}|\lambda_1|, \mathrm{log}|\lambda_2|)$ should lie in the region of
$               \mathrm{log}|\lambda_1|+ \mathrm{log} |\lambda_2 | < C_{\mathrm{MSC_{BD}}} $. The necessary requirement is given by the following region in $(\mathrm{log}|\lambda_1|,\mathrm{log}|\lambda_2|)$ plane
           \begin{equation*}
           \left\{
          \begin{aligned}
            &\log{|\lambda_1|}  < C_{\mathrm{MSC_{BD}}}, \;\;  \log { |\lambda_2 |} < C_{\mathrm{MSC_{BD}}} \\
             & \log {|\lambda_1|} +\log{| \lambda_2 |} < - \log{{\big( \epsilon+(1-\epsilon) \big( \frac{\sigma_n^2}{\sigma_n^2+P} \big)^{\frac{1}{2}} \big) }}
          \end{aligned}\right.
        \end{equation*}
The necessary and sufficient condition to ensure mean square stability using linear encoders/decoders for this system is given in \cite{Xiao2011}, which states that $(\mathrm{log}|\lambda_1|, \mathrm{log}|\lambda_2|)$ should be in the region constrained by
$      \mathrm{log}|\lambda_1|+ \mathrm{log} |\lambda_2 | <  C_{\mathrm{MSL_{BD}}}$.
Selecting $P = 1$, $\sigma_n^2 = 1$ and $\epsilon = 0.8$, we can plot the regions for $(\mathrm{log}|\lambda_1|, \mathrm{log}|\lambda_2|)$ indicated by the sufficiency and necessity in Theorem \ref{theorem:VectorResult} and that indicated in Theorem 3.1 in \cite{Xiao2011} in Fig. \ref{Fig.vectorSystemStabilityRegion}. We can observe that the region of $(\mathrm{log}|\lambda_1|, \mathrm{log}|\lambda_2|)$ that can be stabilized with the designed causal encoders/decoders in Section IV is much larger than that can be stabilized by linear encoders/decoders in \cite{Xiao2011}. Thus by extending endocers/decoders from linear settings to causal requirements, we can tolerate more unstable systems.
\begin{figure}
  \centering
  \includegraphics[width=0.29\textwidth]{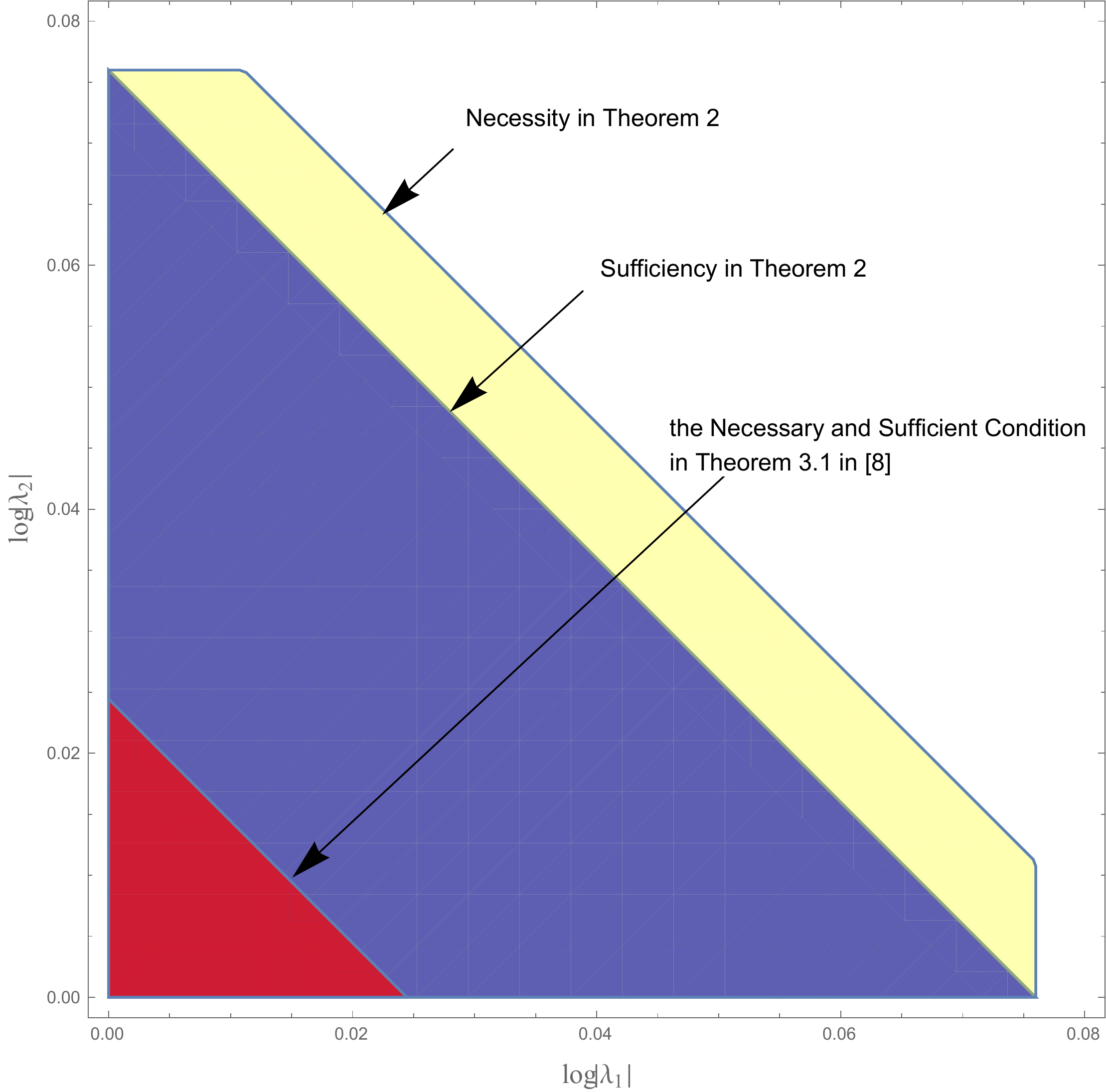}\\
  \caption{Stability region of $(\mathrm{log}|\lambda_1|, \mathrm{log}|\lambda_2|)$ indicated by Theorem 2 for a vector system}\label{Fig.vectorSystemStabilityRegion}
\end{figure}

\section{Conclusion}
This paper characterized the requirement for a power constrained fading channel to allow the existence of a causal encoder/decoder pair that can mean square stabilize a discrete-time LTI system. The mean square capacity of the power constrained fading channel with causal encoders/decoders was given. It was shown that this mean square capacity is smaller than the Shannon capacity and they coincide with each other for some special situations.  Throughout the paper, the capacity was derived with the assumption that there exists a perfect feedback link from the channel output to the channel input. What would the capacity be for power constrained fading channels when there is no such feedback link  or there is only a noisy feedback link is still under investigation.
\bibliographystyle{ieeetr}
\bibliography{reference}

\end{document}